\documentclass[11pt]{amsart}
\usepackage{amssymb,amsmath,amsthm,enumitem,tikz,graphicx}
\theoremstyle{plain}
\newtheorem{theorem}{Theorem}[section]
\newtheorem{corollary}[theorem]{Corollary}
\newtheorem{lemma}[theorem]{Lemma}

\theoremstyle{remark}
\newtheorem*{remark}{Remark}
\newtheorem*{remarks}{Remarks}

\newcommand{\CC}{{\mathbb C}}
\newcommand{\RR}{{\mathbb R}}

\renewcommand{\Re}{\operatorname{Re}}

\begin{document}

\title{Norms of polynomials of the Volterra operator}

\date{29 July 2022}

\author{Thomas Ransford}
\address{D\'epartement de math\'ematiques et de statistique, Universit\'e Laval,
Qu\'ebec  (QC)  G1V 0A6, Canada}
\email{thomas.ransford@mat.ulaval.ca}

\author{Nathan Walsh}
\address{D\'epartement de math\'ematiques et de statistique, Universit\'e Laval, Qu\'ebec (QC) G1V 0A6, Canada}
\email{nathan.walsh.2@ulaval.ca}

\thanks{First author supported by grants from NSERC and the Canada Research Chairs program.
Second author supported by an NSERC Undergraduate Student Research Award and an FRQNT Supplement}

\begin{abstract}
We compute the operator norm of real-quadratic polynomials of the Volterra operator.
This is used to test whether the Crouzeix conjecture holds for the Volterra operator.
\end{abstract}



\maketitle


\section{Introduction}\label{S:intro}

Let $V: L^2[0,1]\to L^2[0,1]$ be the  Volterra operator, defined by
\[
Vf(x):=\int_0^x f(t)\,dt.
\]
It is well known that $V$ is a compact  quasi-nilpotent operator, and that its adjoint is given by
\[
V^*f(x)=\int_x^1 f(t)\,dt.
\]

Halmos \cite[Problem~188]{Ha82} computed the exact value of the operator norm of $V$. It is given by
\begin{equation}\label{E:Halmos}
\|V\|=2/\pi.
\end{equation}
In fact his method yields all the singular values of $V$. The basic idea is to convert the eigenvalue problem
$V^*Vf=\sigma^2f$ into a second-order ODE with two boundary conditions, which can then be solved explicitly.
We shall refer to this technique as the Halmos method.

Thorpe \cite{Th98} discussed how to extend Halmos's method to higher powers of $V$,
and computed $\|V^n\|$ numerically for certain values of $n$ as far as $n=20$.
The values of $\|V^n\|$ for $1\le n\le 10$ had also previously been computed by Lao and Whitley \cite{LW97},
who further conjectured that the powers of $V$ satisfy the asymptotic estimate $\|V^n\|\sim 1/(2n!)$ as $n\to\infty$.
This conjecture was proved by several people at around the same time: 
Thorpe \cite{Th98}, Little and Reade \cite{LR98}, and Kershaw \cite{Ke99}.
To our knowledge, the  best  upper and lower bounds for $\|V^n\|$ are due to
B\"ottcher and D\"orfler \cite{BD09}.

The problem of computing  $\|p(V)\|$ for more general polynomials $p$
was addressed by Lyubich and Tsedenbayar in \cite{LT10}.
They used the Halmos method to determine all the 
singular values of the linear polynomials $I+\nu V$, where $\nu\in\CC$.
They also explained how the method could be adapted, in principle, to treat polynomials of higher
degree, although they did not carry this  out in detail.

In this article, we begin in \S\ref{S:linear}  by revisiting the case of linear polynomials.
Our aim is to bring to light certain aspects that were not treated in \cite{LT10}.
Then, in \S\ref{S:quadratic}, we turn to quadratic polynomials. Here we carry out the 
program proposed in \cite{LT10} to compute the norm of $\|p(V)\|$ when $p$ is a real-quadratic polynomial.
The computations reveal several interesting aspects of these norms.

This research is motivated, in part, by a question posed to the first author by
Felix Schwenninger, as to whether the Crouzeix conjecture holds for the Volterra operator.
Namely, 
does every polynomial $p$ satisfy $\|p(V)\|\le 2\max_{W(V)}|p(z)|$, where $W(V)$ denotes the
numerical range of $V$?
In \S\ref{S:nr} we use our results from the preceding sections
to perform some numerical calculations to test the conjecture.


\section{Linear polynomials}\label{S:linear}

We begin by stating a formula for $\|p(V)\|$ when $p(z)=z+\mu$, where $\mu\in\CC$.

\begin{theorem}\label{T:linear}
Let $\mu=\alpha+i\beta\in\CC$. Then
\[
\|V+\mu I\|=\sqrt{\alpha^2+1/\rho^2},
\]
where $\rho$ is the unique positive number such that  
\[
\rho/(1-\beta^2\rho^2)\in(0,\pi)
\quad\text{and}\quad
\cot\Bigl(\frac{\rho}{1-\beta^2\rho^2}\Bigr)=\alpha\rho.
\]
\end{theorem}

\begin{proof}
As mentioned in the introduction, Lyubich and Tsedenbayar used the Halmos method
to determine the singular values of $I+\nu V$. The details are somewhat lengthy, and we do
not reproduce them here, referring the interested reader to \cite[Theorem~2.2]{LT10}.
In particular, the largest singular value gives the norm of $I+\nu V$, see \cite[Theorem~2.3]{LT10}.
Using the elementary relation $\|V+\mu I\|=|\mu|\|I+(1/\mu)V\|$,
we deduce the formula in the statement of the theorem.
\end{proof}


In the rest of the section, we explore some consequences of Theorem~\ref{T:linear}.
Figure~\ref{F:contour1} is a contour plot of the function $\mu\mapsto\|V+\mu I\|$, computed using the theorem.
As with all the plots in this article, the computations were performed using Haskell,
and the graphics produced using Gnuplot.

\begin{figure}[ht]
\begin{center}
\includegraphics[scale=0.18, trim= 0 100 0 50]{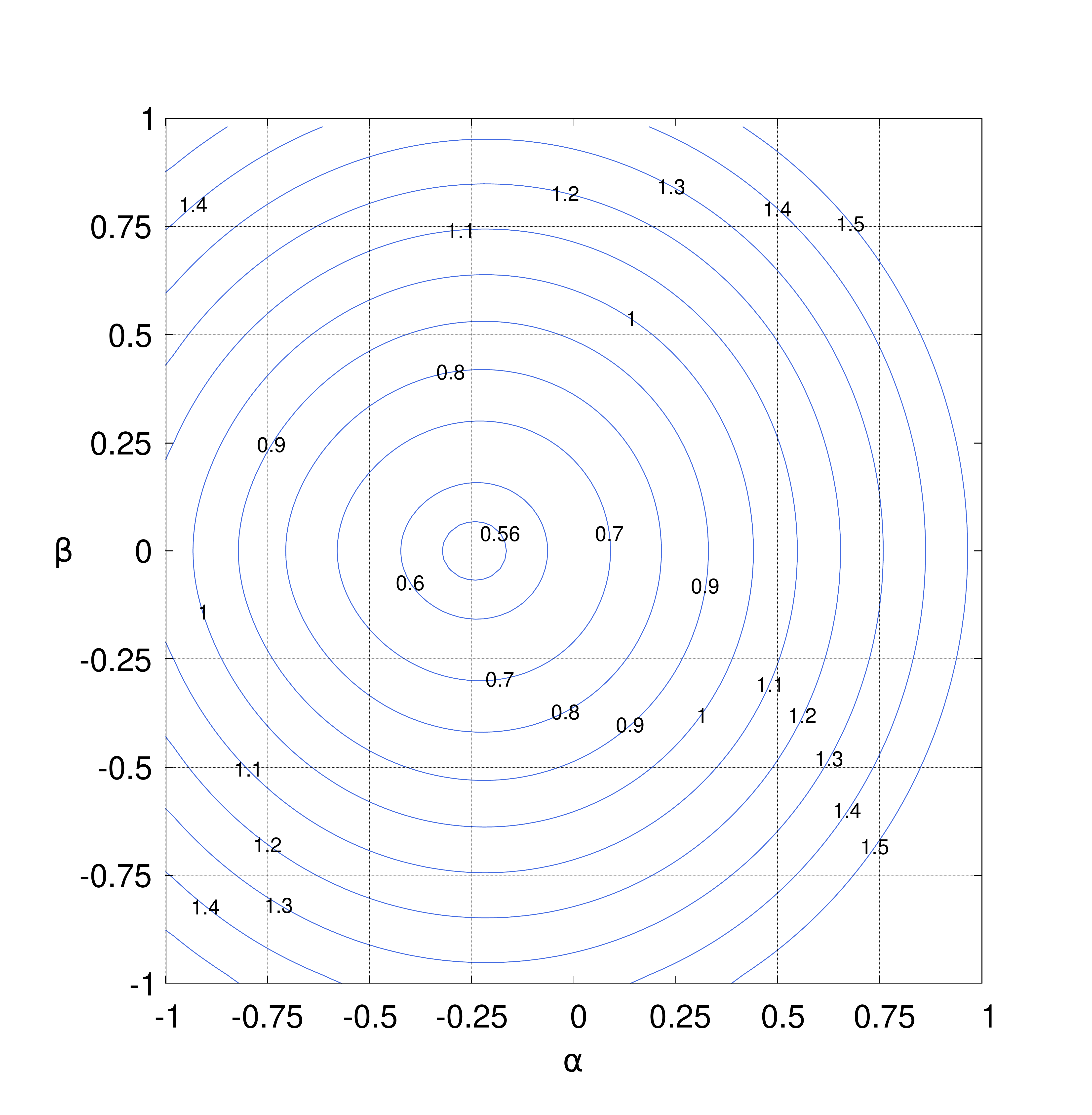}
\caption{Contour plot of the function $\mu\mapsto\|V+\mu I\|$}
\label{F:contour1}
\end{center}
\end{figure}

Several features are immediately apparent. 
First, the function appears to be symmetric about the horizontal axis.
This is obvious from Theorem~\ref{T:linear}, since $\|V+\mu I\|$ depends on $\beta$ only through $\beta^2$.
(It could also have been proved directly by exploiting the fact that $V$ is unitarily equivalent to $V^*$
via $V^*=U^*VU$, where $Uf(x):=f(1-x)$.)

Another obvious feature in Figure~\ref{F:contour1} is the apparent existence of a minimum,
attained near the point $-0.25$, with value about $0.55$. The next theorem confirms this observation.

\begin{theorem}\label{T:minimum}
The minimum of $\|V+\mu I\|$ is attained at a unique point $\mu_0\approx-0.2429626850$
and has value $\sqrt{\mu_0^2-\mu_0}\approx0.5495393994$.
The exact value of $\mu_0$ is $-1/\rho_0^2$, 
where $\rho_0$ is the unique solution in $(\pi/2,\pi)$ of the equation $\rho_0+\tan\rho_0=0$.
\end{theorem}

\begin{proof}
The function $\mu\mapsto\|V+\mu I\|$ is continuous, and tends to infinity as $|\mu|\to\infty$, so it attains a minimum
somewhere in the complex plane.

It is a simple consequence of Theorem~\ref{T:linear} that, given $\alpha\in\RR$, 
we have  $\|V+(\alpha+i\beta_1)I\|=\|V+(\alpha+i\beta_2)I\|$ if and only if $\beta_1=\pm\beta_2$. 
It follows that the map $\beta\mapsto\|V+(\alpha+i\beta)I\|$ is injective on each of the intervals
$[0,\infty)$ and $(-\infty,0]$. Since it is also convex, it must be strictly increasing on $[0,\infty)$ and
strictly decreasing on $(-\infty,0]$. This shows that the map $\mu\mapsto\|V+\mu I\|$ 
can only attain its minimum on the real axis.

So we seek to minimize $\|V+\alpha I\|$ for $\alpha\in\RR$.
By Theorem~\ref{T:linear}, applied with $\beta=0$, this amounts to minimizing
$\alpha^2+1/\rho^2$ subject to the constraints $\rho\in(0,\pi)$ and $\cot\rho=\alpha\rho$. 
We solve this problem using Lagrange multipliers. Consider the Lagrangian 
\[
L(\alpha,\rho,\lambda):=\alpha^2+1/\rho^2-\lambda(\cot\rho-\alpha\rho),
\]
where $\alpha,\lambda\in\RR$ and $\rho\in(0,\pi)$.
Equating the partial derivatives $\partial L/\partial \alpha,\partial L/\partial \rho,\partial L/\partial \lambda$ 
to zero yields the system of equations
\[
\left\{
\begin{aligned}
2\alpha+\lambda\rho&=0,\\
-2/\rho^3+\lambda\csc^2\rho+\lambda\alpha &=0,\\
\cot\rho-\alpha\rho&=0.
\end{aligned}
\right.
\]
Eliminating $\lambda$ and $\alpha$ leads to the equation $\rho+\tan\rho=0$,
which has a unique solution $\rho_0$ in $(0,\pi)$. Since $\tan\rho_0=-\rho_0<0$,
we see that in fact $\rho_0\in(\pi/2,\pi)$. Substituting back into the equation $\cot\rho-\alpha\rho=0$,
we get $\alpha=-1/\rho_0^2$, and hence $\|V+\alpha I\|=\sqrt{1/\rho_0^4+1/\rho_0^2}$.
\end{proof}

It is perhaps tempting to believe that the function $\alpha\mapsto\|V+\alpha I\|$ is symmetric about $\alpha=\mu_0$,
but in fact this is not true. For example, a calculation using Theorem~\ref{T:linear} gives
\[
\|V+(\mu_0+5)I\|=5.265379338\dots 
\quad\text{and}\quad
\|V+(\mu_0-5)I\|=5.253007667\dots .
\]

The behaviour of $\|V+\mu I\|$ on the imaginary axis is more straightforward.
There is even an explicit expression for $\|V+i\beta I\|$, generalizing Halmos's  result \eqref{E:Halmos}.
The following result is essentially \cite[Corollary~2.4]{LT10}.

\begin{theorem}\label{T:imagnorm}
For all $\beta\in\RR$,
\[
\|V+i\beta I\|=\frac{1}{\pi}+\sqrt{\beta^2+\frac{1}{\pi^2}}.
\]
\end{theorem}

\begin{proof}
Applying Theorem~\ref{T:linear} with $\alpha=0$, we have $\|V+i\beta I\|=1/\rho$,
where $\rho$ is the unique positive solution of
\[
\frac{\rho}{1-\beta^2\rho^2}\in(0,\pi)
\quad\text{and}\quad
\cot\Bigl(\frac{\rho}{1-\beta^2\rho^2}\Bigr)=0.
\]
These conditions imply that
$\rho/(1-\beta^2\rho^2)=\pi/2$, or, equivalently, $(1/\rho)^2-(2/\pi)(1/\rho)-\beta^2=0$.
Solving this equation for $1/\rho$ and substituting into the formula for $\|V+i\beta I\|$ gives the result.
\end{proof}

It is also instructive to consider the function $\| I+\nu V\|$ for $\nu\in\CC$,
which can easily be calculated using Theorem~\ref{T:linear}.
Figure~\ref{F:contour2} is  a contour plot of the function $\nu\mapsto\|I+\nu V\|$.

\begin{figure}[ht]
\begin{center}
\includegraphics[scale=0.18, trim= 0 100 0 50]{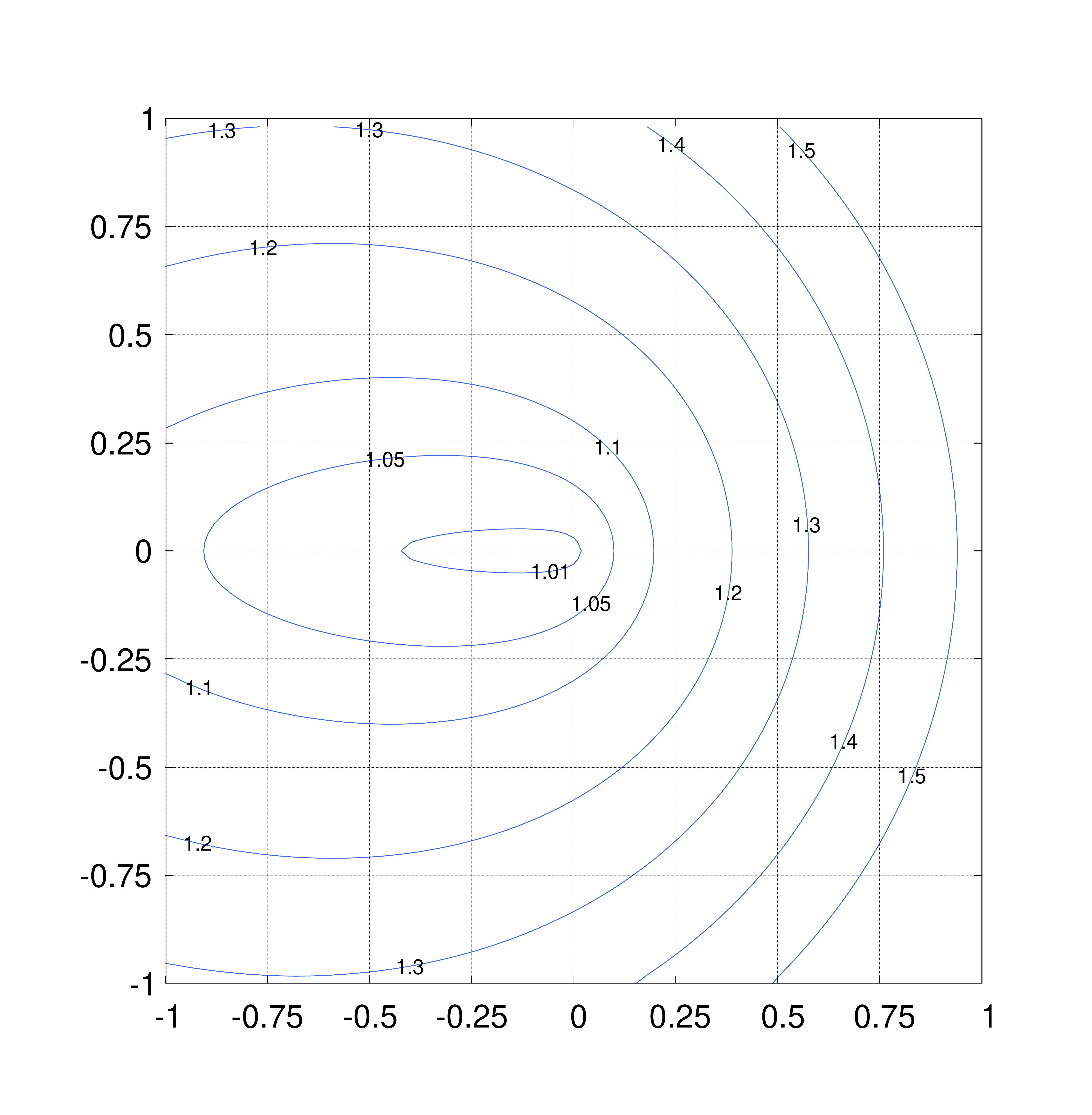}
\caption{Contour plot of the function $\nu\mapsto\|I+\nu V\|$}
\label{F:contour2}
\end{center}
\end{figure}

In this case, the minimum is attained at $\nu=0$. 
Indeed, since the operator norm is always at least as large as the spectral radius,
we clearly have $\|I+\nu V\|\ge 1$ for all $\nu\in\CC$. 
The following result,
due to Lyubich and Tsedenbayar \cite[Corollary~2.5]{LT10},
 shows that, in fact, the minimum is attained only at $0$.

\begin{theorem}\label{T:strict}
$\|I+\nu V\|>1$ for all $\nu\in\CC\setminus\{0\}$.
\end{theorem}

\begin{proof}
Let $\mu=\alpha+i\beta$.
By Theorem~\ref{T:linear}, $\|V+\mu I\|=\sqrt{\alpha^2+1/\rho^2}$,
for some  positive $\rho$ with $\rho/(1-\beta^2\rho^2)\in(0,\pi)$.
Then $\beta^2\rho^2<1$, so $\|V+\mu I\|>\sqrt{\alpha^2+\beta^2}=|\mu|$.
The result follows.
\end{proof}

Lyubich and Tsedenbayar 
conjectured that, if $p(z)$ is any non-constant polynomial with $p(0)=1$,
then $\|p(V)\|>1$. Their conjecture was later refuted by ter Elst and Zem\'anek \cite{tZ18}.
We shall return to this topic in the next section.

The last theorem in the present section describes the asymptotic behaviour of $\|I+\nu V\|$ as $\nu\to0$.
It will play a role when we come to treat the numerical range of $V$ in \S\ref{S:nr}.

\begin{theorem}\label{T:nearI}
For  $\theta\in[-\pi,\pi]$, 
\begin{equation}\label{E:nearI}
\|I+r e^{i\theta} V\|=1+r\frac{\sin\theta}{2\theta}+O(r^2)
\quad(r\to0^+).
\end{equation}
\end{theorem}

\begin{proof}
By Theorem~\ref{T:linear},
\begin{equation}\label{E:normformula}
\|I+re^{i\theta}V\|=\sqrt{\cos^2\theta+r^2/\rho^2},
\end{equation}
where
\begin{equation}\label{E:conds}
\phi:=\frac{\rho}{1-(\rho^2/r^2)\sin^2\theta}\in(0,\pi)
\quad\text{and}\quad
\cot\phi=(\rho/r)\cos\theta.
\end{equation}

If $\theta=0$, then the conditions in \eqref{E:conds} simplify to $\phi=\rho$ and  $\rho\tan\rho=r$. 
Then, as  $r\to0^+$, we have $\rho=\sqrt{r}(1+O(r))$. Feeding this information back into \eqref{E:normformula},
we obtain 
\[
\|I+re^{i\theta}V\|=\sqrt{1+r+O(r^2)}=1+\frac{r}{2}+O(r^2) \quad(r\to0^+).
\]

If $\theta=\pm\pi$, then then the conditions in \eqref{E:conds} simplify to $\phi=\rho$ and  $\rho\tan\rho=-r$. 
Then, as  $r\to0^+$, we have $\rho=\pi/2+O(r)$. Feeding this information back into \eqref{E:normformula},
we obtain 
\[
\|I+re^{i\theta}V\|=\sqrt{1+O(r^2)}=1+O(r^2) \quad(r\to0^+).
\]

In the rest of the argument, we assume that $\sin\theta\ne0$.
Since $\|I+re^{i\theta}V\|=1+O(r)$ as $r\to0^+$, the formula \eqref{E:normformula}
implies that $r^2/\rho^2=\sin^2\theta+O(r)$, and therefore, from \eqref{E:conds},
\[
\frac{\cot^2\theta}{\cot^2\phi}=\frac{r^2}{\rho^2}\frac{1}{\sin^2\theta}=1+O(r).
\]
This implies that $\phi=|\theta|+O(r)$.
Hence, as $r\to0^+$,
\begin{align*}
\frac{r^2}{\rho^2}-\sin^2\theta&=\frac{r^2}{\rho^2}\Bigl(1-(\rho^2/r^2)\sin^2\theta\Bigr)
=\frac{r^2}{\rho^2}\frac{\rho}{\phi}=\frac{r}{\phi}\frac{r}{\rho}\\
&=\frac{r}{|\theta|+O(r)}(|\sin\theta| +O(r))
=r\frac{\sin\theta}{\theta}+O(r^2).
\end{align*}
Feeding this into \eqref{E:normformula}, we obtain
\[
\|I+re^{i\theta}V\|=\sqrt{\cos^2\theta+\sin^2\theta+r\frac{\sin\theta}{\theta}+O(r^2)}
=1+r\frac{\sin\theta}{2\theta}+O(r^2).\qedhere
\]
\end{proof}


\section{Quadratic polynomials}\label{S:quadratic}

Our aim in this section is to compute the operator norm of $p(V)=V^2+\sigma V+\tau I$, 
where $\sigma,\tau\in\RR$.
We restrict ourselves to real coefficients, 
because the case of general complex coefficients rapidly becomes
prohibitively complicated.

The general strategy is the one outlined in \cite{LT10}, once again following 
the Halmos method.
Since $p(V)^*p(V)$ is both positive and of the form ($\tau^2 I+$ compact), 
its spectrum consists of $\tau^2$ together with a 
sequence of non-negative eigenvalues converging to $\tau^2$. 
We therefore try to identify these eigenvalues.
Since our interest is in the spectral radius, 
we can restrict ourselves to considering eigenvalues that are greater than $\tau^2$.
 
So let $\lambda$ be an eigenvalue of $p(V)^*p(V)$ such that $\lambda>\tau^2$, 
say $\lambda=\tau^2+\delta$, where $\delta>0$.
Then there exists $f\in L^2[0,1]$ with $f\not\equiv0$ such that 
\begin{equation}\label{E:evalueq}
(V^2+\sigma V+\tau I)^*(V^2+\sigma V+\tau I)f=(\tau^2+\delta) f.
\end{equation}
Since $V$ and $V^*$ both map $L^2[0,1]$ into $C[0,1]$ and $\delta>0$, 
we have $f\in C[0,1]$. By bootstrapping,
it follows that $f\in C^\infty[0,1]$.
Let $D$ denote the differentiation operator. By the fundamental theorem of calculus, $DV=I$ and $DV^*=-I$.
Applying $D^4$ to \eqref{E:evalueq} and simplifying, we obtain the differential equation
\begin{equation}\label{E:odeq}
\delta f^{(4)}(x)+(\sigma^2-2\tau)f^{(2)}(x)- f(x)=0.
\end{equation}
A short calculation shows that $V^2f(x)=\int_0^x(x-t)f(t)\,dt$.
Since  $VF(0)=0$ and $(V^*F)(1)=0$ for all $F\in L^2[0,1]$, 
we obtain the following boundary conditions:
\begin{equation}\label{E:bcq}
\left\{
\begin{aligned}
\sigma\tau\int_0^1 f(t)\,dt+\tau\int_0^1(1-t)f(t)\,dt&=\delta f(1),\\
(\tau-\sigma^2)\int_0^1 f(t)\,dt-\sigma\int_0^1(1-t)f(t)\,dt &=\delta f'(1),\\
(2\tau-\sigma^2)f(0)&=\delta f''(0),\\
(2\tau-\sigma^2)f'(0)&=\delta f^{(3)}(0).
\end{aligned}
\right.
\end{equation}

We try a solution of the form $e^{\omega x}$. 
This will indeed be a solution of \eqref{E:odeq} provided that $\omega$ 
satisfies the quartic equation
\[
\delta\omega^4+(\sigma^2-2\tau)\omega^2-1=0,
\]
in other words, if $\omega$ is the square root of a solution $t$ of the quadratic equation
\[
\delta t^2+(\sigma^2-2\tau)t-1=0.
\]
Since the coefficient of $t^2$ is strictly positive and the constant coefficient is strictly negative, 
this quadratic equation has two real roots, one positive, the other negative. It follows that
the quartic equation has roots $\pm\omega_1$ and $\pm i\omega_2$, where $\omega_1,\omega_2>0$.
After rearranging slightly, we deduce that the general solution to \eqref{E:odeq} is
\begin{equation}\label{E:odesolnq}
f(x)=A_1\cosh\omega_1x+B_1\sinh\omega_1x+A_2\cos \omega_2x+B_2\sin\omega_2x,
\end{equation}
where $\omega_1,\omega_2$ are the unique positive solutions to
\begin{equation}\label{E:omegaq}
\left\{
\begin{aligned}
\delta\omega_1^4+(\sigma^2-2\tau)\omega_1^2-1&=0,\\
\delta\omega_2^4-(\sigma^2-2\tau)\omega_2^2-1&=0.
\end{aligned}
\right.
\end{equation}

Plugging the function $f(x)$ from \eqref{E:odesolnq} into the last two boundary conditions 
in \eqref{E:bcq}, 
we obtain 
\begin{align*}
(2\tau-\sigma^2)(A_1+A_2)&=\delta(\omega_1^2A_1-\omega_2^2A_2),\\
(2\tau-\sigma^2)(\omega_1B_1+\omega_2B_2)&=\delta(\omega_1^3B_1-\omega_2^3B_2),
\end{align*}
which, upon using the relations \eqref{E:omegaq} 
to eliminate $\delta$, simplify to
\begin{align*}
A_1/\omega_1^2&=A_2/\omega_2^2,\\
B_1/\omega_1&=B_2/\omega_2.
\end{align*}
If we call these two quantities $A$ and $B$ respectively, we deduce that $f(x)=Ag(x)+Bh(x)$, where
\begin{align*}
g(x)&:=\omega_1^2\cosh\omega_1x+\omega_2^2\cos \omega_2x,\\
h(x)&:=\omega_1\sinh\omega_1x+\omega_2\sin\omega_2x.
\end{align*}
Routine calculations yield
\begin{align*}
g(1)&=\Omega_2, &g'(1)&=\Omega_3, &\int_0^1 g&=\Omega_1, &\int_0^1 (1-x)g&=\Omega_0,\\
h(1)&=\Omega_1, &h'(1)&=\Omega_2, &\int_0^1 h&=\Omega_0, &\int_0^1(1-x)h&=(\sigma^2-2\tau)\Omega_1+\delta\Omega_3,
\end{align*}
where
\begin{equation}\label{E:Omega}
\left\{
\begin{aligned}
\Omega_0&:=\cosh\omega_1-\cos\omega_2,\\
\Omega_1&:=\omega_1\sinh\omega_1+\omega_2\sin\omega_2,\\
\Omega_2&:=\omega_1^2\cosh\omega_1+\omega_2^2\cos\omega_2,\\
\Omega_3&:=\omega_1^3\sinh\omega_1-\omega_2^3\sin\omega_2.\\
\end{aligned}
\right.
\end{equation}
(In the case of $\int_0^1(1-x)h$, we use
\eqref{E:omegaq} to express $\omega_1^{-1}\sinh\omega_1-\omega_2^{-1}\sin\omega_2$ 
 in terms of $\Omega_1$ and $\Omega_3$.)
Thus, plugging $f=Ag+Bh$ into the first two boundary conditions \eqref{E:bcq}, we obtain that
\begin{align*}
\sigma\tau(A\Omega_1+B\Omega_0)+\tau(A\Omega_0+B((\sigma^2-2\tau)\Omega_1+\delta\Omega_3))-\delta(A\Omega_2+B\Omega_1)&=0,\\
(\tau-\sigma^2)(A\Omega_1+B\Omega_0)-\sigma(A\Omega_0+B((\sigma^2-2\tau)\Omega_1+\delta\Omega_3))-\delta(A\Omega_3+B\Omega_2)&=0.
\end{align*}
After rearrangement, these become the linear $2\times 2$ system
\[
\begin{pmatrix}
\tau\Omega_0+\sigma\tau\Omega_1-\delta\Omega_2
&\sigma\tau\Omega_0+(\sigma^2\tau-2\tau^2-\delta)\Omega_1+\tau\delta\Omega_3\\
-\sigma\Omega_0+(\tau-\sigma^2)\Omega_1-\delta\Omega_3
&(\tau-\sigma^2)\Omega_0-\sigma(\sigma^2-2\tau)\Omega_1-\delta\Omega_2-\sigma\delta\Omega_3
\end{pmatrix}
\begin{pmatrix}
A\\B
\end{pmatrix}
=
\begin{pmatrix}
0\\0
\end{pmatrix}.
\]
Since $f\not\equiv0$, the constants $A$ and $B$ cannot both be zero, so the $2\times2$ matrix must be singular,
in other words,
\begin{equation}\label{E:bigeqn}
\begin{aligned}
&\Bigl(\tau\Omega_0+\sigma\tau\Omega_1-\delta\Omega_2\Bigr)
\Bigl((\tau-\sigma^2)\Omega_0-\sigma(\sigma^2-2\tau)\Omega_1-\delta\Omega_2-\sigma\delta\Omega_3\Bigr)\\
&-\Bigl(\sigma\tau\Omega_0+(\sigma^2\tau-2\tau^2-\delta)\Omega_1+\tau\delta\Omega_3\Bigr)\Bigl(-\sigma\Omega_0+(\tau-\sigma^2)\Omega_1-\delta\Omega_3\Bigr)=0.
\end{aligned}
\end{equation}
Since the $\Omega_j$ are functions of $\omega_1$ and $\omega_2$, which in turn are functions of $\delta$,
this last relation gives an equation that must be satisfied by $\delta$. Conversely, if $\delta$ is a positive solution
to this equation, then we can work backwards, and deduce that $\lambda=\tau^2+\delta$ is an eigenvalue of $p(V)^*p(V)$.

We record our findings in a theorem.

\begin{theorem}\label{T:quadratic}
Let $p(V):=V^2+\sigma V+\tau I$, where $\sigma,\tau\in\RR$.
For $\delta>0$, define $\omega_j$ via \eqref{E:omegaq} and $\Omega_j$ via \eqref{E:Omega}.
Then $\|p(V)\|=\sqrt{\tau^2+\delta}$, where $\delta$ is the largest positive solution to 
\eqref{E:bigeqn}. If no positive solution exists, then $\|p(V)\|=|\tau|$.
\end{theorem}

There is one case that is simple enough to solve analytically,
namely  $\sigma=\tau=0$, i.e., $p(V)=V^2$.

\begin{corollary}
$\|V^2\|=\gamma_0^{-2}$, where $\gamma_0$ is the smallest positive solution to 
$(\cosh\gamma_0)(\cos\gamma_0)=-1$. Numerically, $\|V^2\|\approx0.2844\dots$
\end{corollary}

\begin{proof}
When $\sigma=\tau=0$, equation \eqref{E:bigeqn} reduces to
$\Omega_2^2-\Omega_1\Omega_3=0$, i.e. 
\[
\Bigl(\omega_1^2\cosh\omega_1+\omega_2^2\cos\omega_2\Bigr)^2
-\Bigl(\omega_1\sinh\omega_1+\omega_2\sin\omega_2\Bigr)\Bigl(\omega_1^3\sinh\omega_1-\omega_2^3\sin\omega_2\Bigr)=0.
\]
Also the two equations \eqref{E:omegaq} give 
$\omega_1=\omega_2=\delta^{-1/4}$.
Substituting  into the previous equation and simplifying,
we obtain 
\[
\cosh(\delta^{-1/4})\cos(\delta^{-1/4})=-1.
\]
By Theorem~\ref{T:quadratic}, we have $\|V^2\|=\sqrt{\delta}$
where $\delta$ is the largest positive solution to this last equation.
The result follows.
\end{proof}

\begin{remark}
This result has been known for a long time.
According to Thorpe \cite{Th98}, 
the result was known to Halmos, who learned of it from A.~Brown.
The earliest reference that we have been able to track down is\ \cite[Table~1]{Ho73}.
\end{remark}

For $(\sigma,\tau)\ne(0,0)$, though we have not been able to make any further progress toward an analytic solution,
the result of Theorem~\ref{T:quadratic} can be used to compute $\|V^2+\sigma V+\tau I\|$ numerically.
The implementation is more complicated than in the case of linear polynomials,
because Theorem~\ref{T:quadratic} is expressed in terms of `the largest positive solution to \eqref{E:bigeqn}',
rather than in terms of the unique solution in some given interval, as was the case in Theorem~\ref{T:linear}.
Moreover, the equation \eqref{E:bigeqn} may have infinitely many solutions, or possibly no solutions at all. 

Since $\|p(V)\|\le \|V^2\|+|\sigma|\|V\|+|\tau|\|I\|<1+|\sigma|+|\tau|$, 
every positive solution $\delta$ of \eqref{E:bigeqn} must satisfy
$\delta< (1+|\sigma|+|\tau|)^2-\tau^2$. 
Further, if there are infinitely many solutions, then they can only accumulate at $0$ (because they correspond to 
eigenvalues of the compact operator $p(V)^*p(V)-\tau^2I$). 
So our strategy is to perform Newton's method a large number times, 
starting from a spread of points in $(0, (1+|\sigma|+|\tau|)^2-\tau^2)$, 
and then conserving the largest root found. 
If no root is found, then we consider that there is no root, and so $\|V^2+\sigma V+\tau I\|=|\tau|$.

Figure~\ref{F:contour3} is a contour plot of the function $(\sigma,\tau)\mapsto\|V^2+\sigma V+\tau I\|$,
computed using Theorem~\ref{T:quadratic}.

\begin{figure}[ht]
\begin{center}
\includegraphics[scale=0.18, trim= 0 100 0 50]{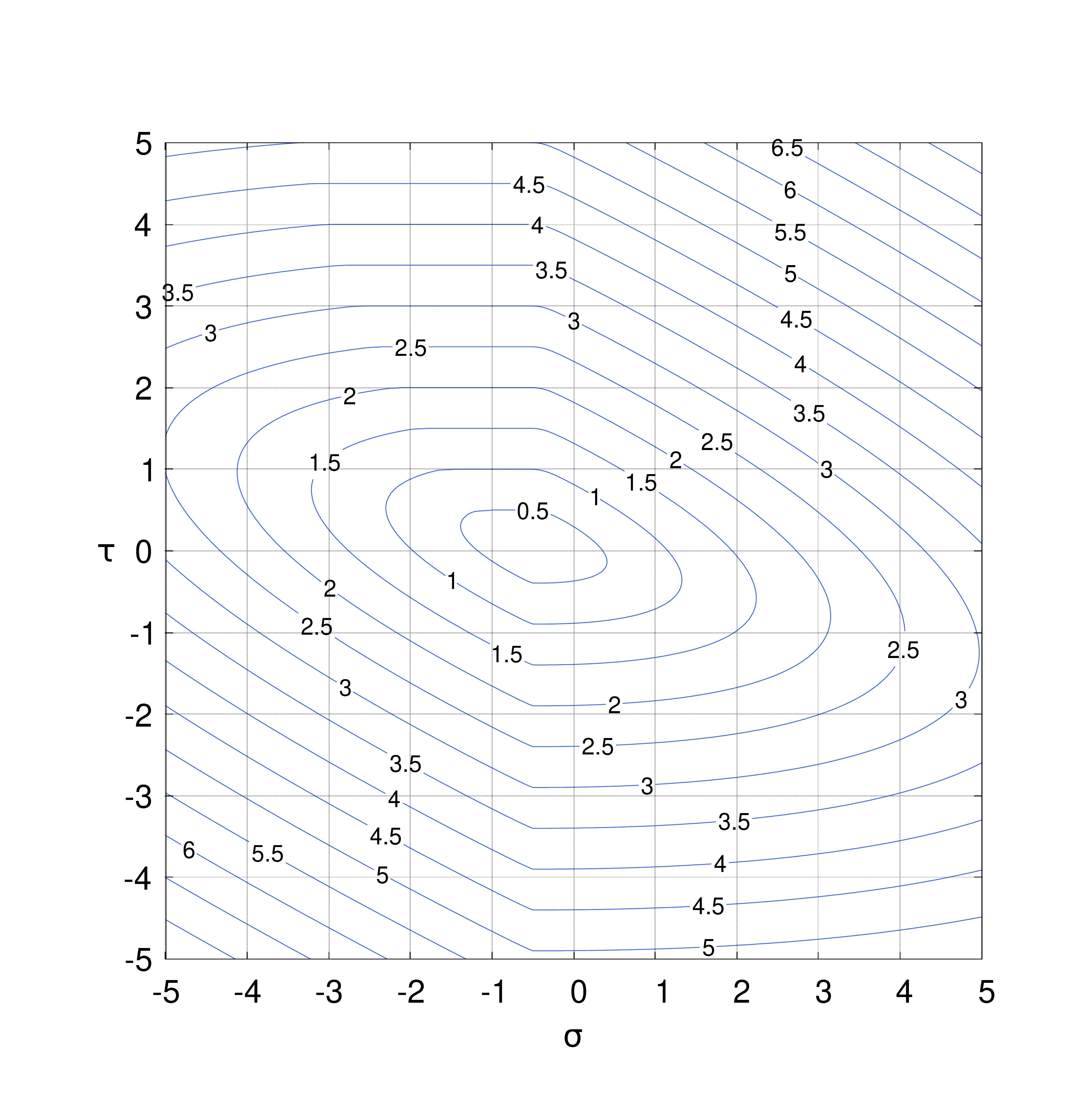}
\caption{Contour plot of the function $(\sigma,\tau)\mapsto\|V^2+\sigma V+\tau I\|$}\label{F:contour3}
\end{center}
\end{figure}

Just as in the case of linear polynomials, it is also instructive to consider $\|I+\xi V+\eta V^2\|$.
Figure~\ref{F:contour4} is a colour plot of this norm as a function of $(\xi,\eta)$.

\begin{figure}[ht]
\begin{center}
\includegraphics[scale=0.09, trim= 0 100 0 50]{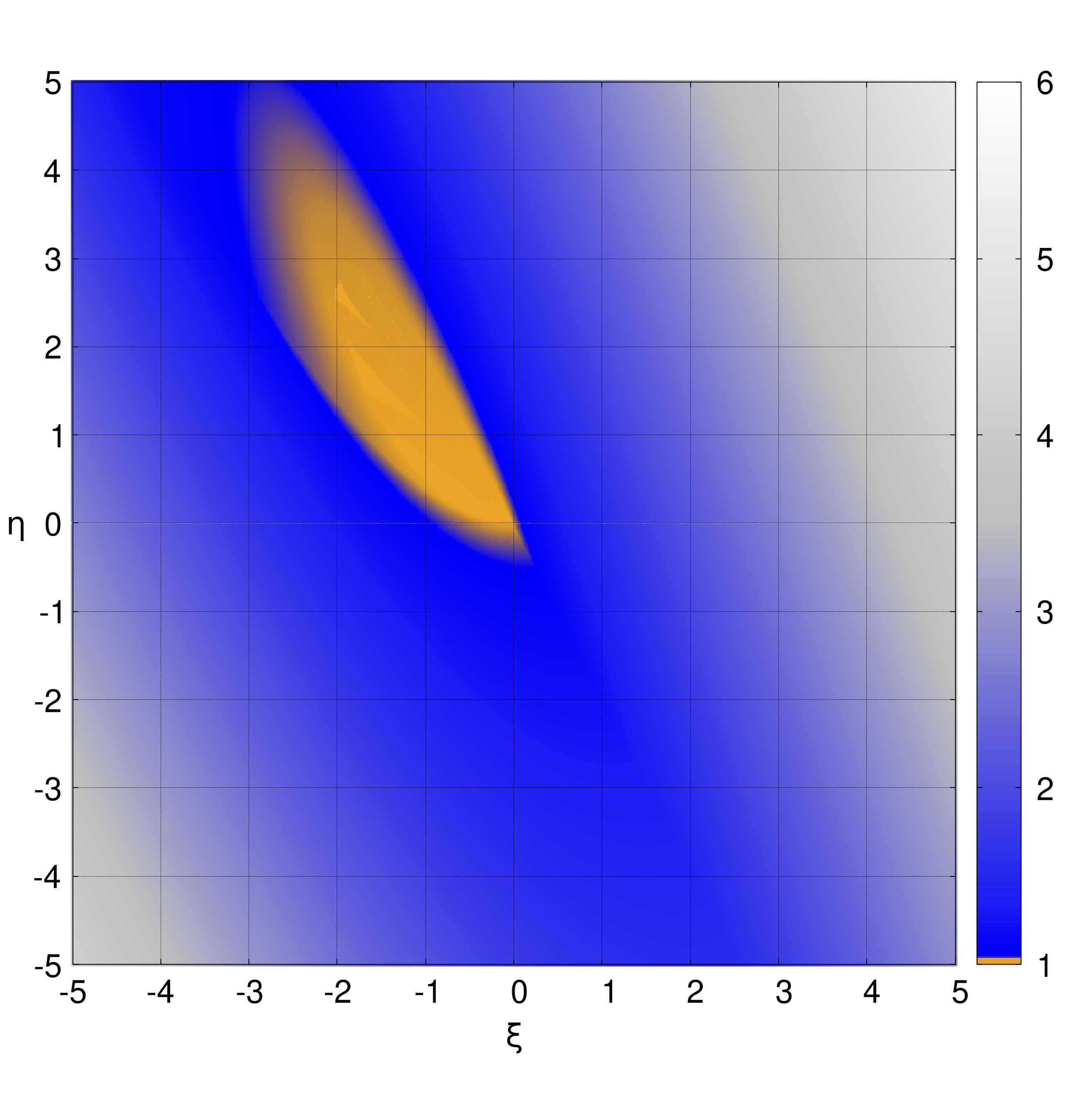}
\caption{Colour plot of the function $(\xi,\eta)\mapsto\|I+\xi V+\eta V^2\|$}\label{F:contour4}
\end{center}
\end{figure}

Of course, we always have $\|I+\xi V+\eta V^2\|\ge1$, 
since the norm is at least as large as the spectral radius.
Figure~\ref{F:contour4} suggests that there is a region of pairs $(\xi,\eta)$ around $(-1,1)$
where $\|I+\xi V+\eta V^2\|=1$.  The following result proves this rigorously.

\begin{theorem}\label{T:flatnorm}
Suppose that $\xi,\eta\in\RR$ satisfy 
\begin{equation}\label{E:flatnorm}
\left\{
\begin{aligned}
&\xi\le 0,\\
&4\eta^2/\pi^2-2\eta+\xi^2\le 0,\\
&(4\xi/\pi^2-1)\eta^2+(\xi^2-2\xi)\eta+\xi^3\ge0.
\end{aligned}\right.
\end{equation}
Then $\|I+\xi V+\eta V^2\|=1$.
In particular, we have $\|I-V+V^2\|=1$.
\end{theorem}

The left-hand  side of Figure~\ref{F:flatnorm} illustrates the range of pairs $(\xi,\eta)$
satisfying the inequalities \eqref{E:flatnorm}.
The right-hand side of Figure~\ref{F:flatnorm} superimposes this set on 
the colour plot of the function $(\xi,\eta)\mapsto\|I +\xi V+\eta V^2\|$
previously obtained.

\begin{figure}[ht]
\begin{minipage}{0.49\linewidth}
\begin{center}
\includegraphics[scale=0.4, trim=0 330 100 10, clip]{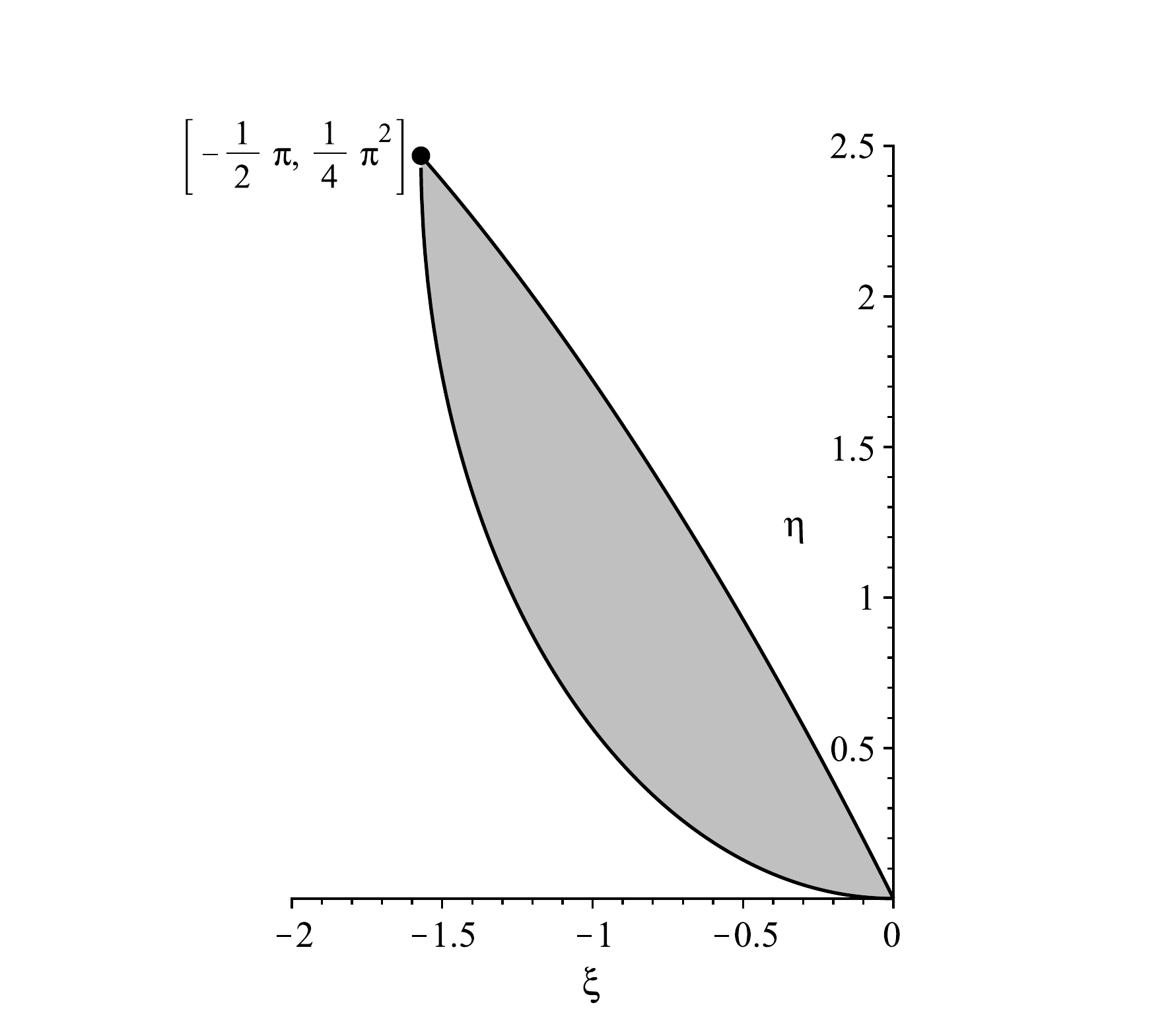}
\end{center}
\end{minipage}
\begin{minipage}{0.49\linewidth}
\begin{center}
\includegraphics[scale=0.07, trim= 0 120 0 50]{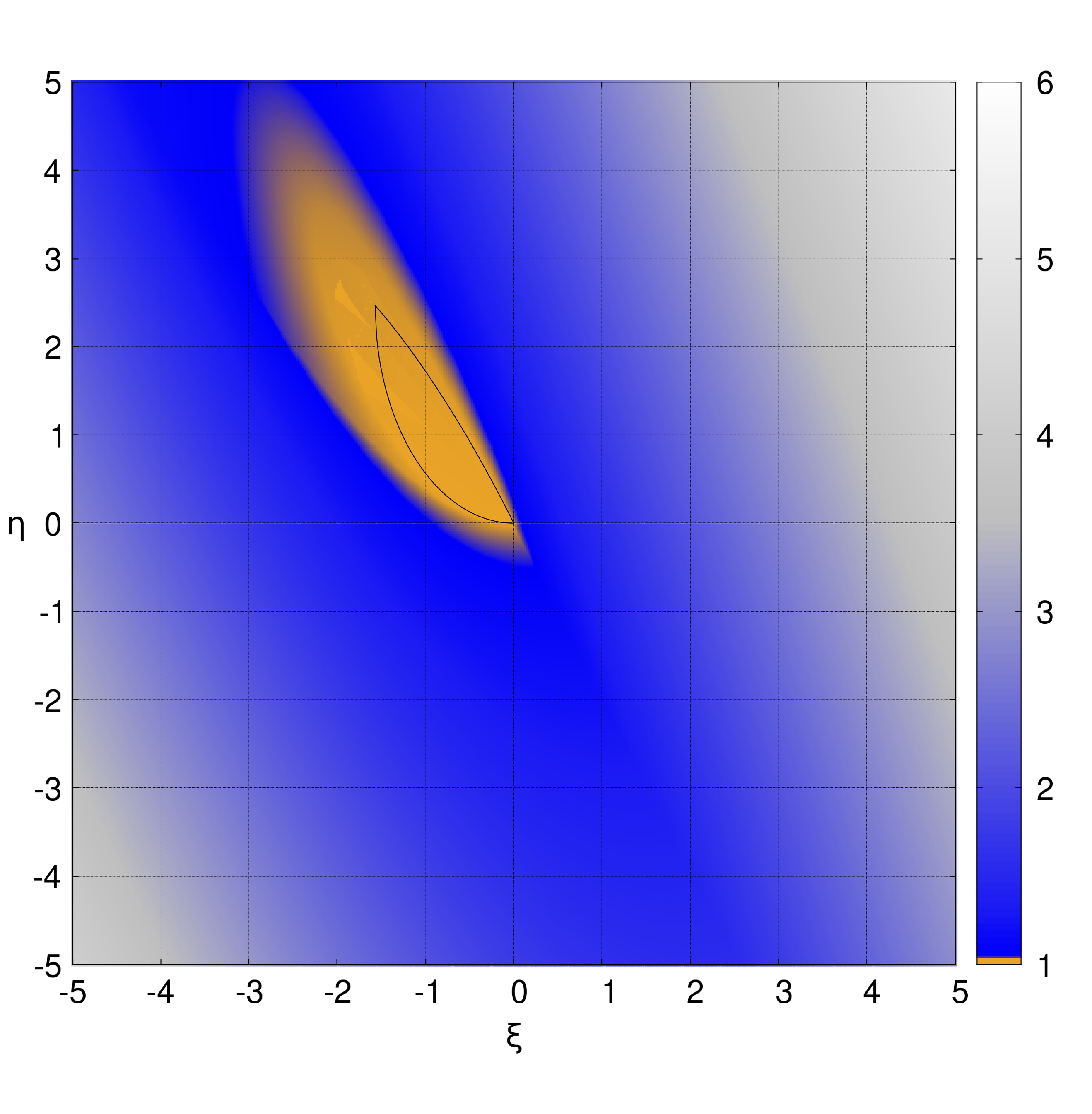}
\end{center}
\end{minipage}
\caption{Range of pairs ($\xi,\eta$) satisfying \eqref{E:flatnorm}} 
\label{F:flatnorm}
\end{figure}

\begin{remarks}
(1) As mentioned in \S\ref{S:linear}, it was conjectured in \cite{LT10} that $\|p(V)\|>1$ whenever $p(z)$
is a non-constant polynomial with $p(0)=1$. Theorem~\ref{T:flatnorm} disproves this conjecture.
It is not the first such example: ter Elst and Zem\'anek \cite{tZ18}, using a different method, showed that 
$\|1-bV+bV^2\|=1$ for all $b\in[0,1/8]$. However, they were unable to show that this equality extends to $b=1$
(see \cite[Remark~3.3]{tZ18}). Not only does our theorem establish this, but it provides a whole open
region of coefficients $(\xi,\eta)$ around $(-1,1)$ for which $\|I+\xi V+\eta V^2\|=1$.

(2) As remarked in \cite{tZ18}, the set of polynomials $p(z)$ with the property that $\|p(V)\|=p(0)=1$ is both a convex set and a multiplicative semigroup. Thus, the quadratic polynomials that arise in Theorem~\ref{T:flatnorm} can be used to generate a much larger set of polynomials $p$ with $\|p(V)\|=p(0)=1$.
\end{remarks}

To prove Theorem~\ref{T:flatnorm}, we make use of the following elementary  lemma.

\begin{lemma}\label{L:completesquare}
Let $r,s,t\in\RR$. The following are equivalent:
\begin{enumerate}[label=\rm({\alph*})]
\item $r|z|^2+s|w|^2+t (z\overline{w}+\overline{z}w)\ge0$ for all $z,w\in\CC$;
\item $r,s\ge0$ and $rs\ge t^2$.
\end{enumerate}
\end{lemma}

\begin{proof}
For $r\ne0$, this becomes clear upon completing the square:
\[
r |z|^2+s|w|^2+t (z\overline{w}+\overline{z}w)
=\frac{1}{r}\Bigl( |r z+t w|^2+(rs-t^2)|w|^2\Bigr).
\]
Likewise if $s\ne0$. Finally, if $r=s=0$, then both (a) and (b) are equivalent to the 
condition that $t=0$.
\end{proof}

\begin{proof}[Proof of Theorem~\ref{T:flatnorm}]
Fix $\xi,\eta\in\RR$. We seek conditions on $\xi,\eta$ ensuring that
\[
\|(I+\xi V+\eta V^2)f\|_2\le \|f\|_2^2 \quad(f\in C[0,1]).
\]
It suffices to consider $f\in C[0,1]$, since this is a dense subset of $L^2[0,1]$.

So let $f\in C[0,1]$, and set $g:=V^2f$. 
Then $g\in C^2[0,1]$ with $g''=f$ and $g(0)=g'(0)=0$.
We need to determine when
\begin{equation}\label{E:gineq}
\int_0^1|g''+\xi g'+\eta g|^2\le \int_0^1|g''|^2.
\end{equation}
Now, a calculation gives
\begin{align*}
&\int_0^1|g''+\xi g'+\eta g|^2\\
&=\int_0^1|g''|^2+\xi^2\int_0^1|g'|^2+\eta^2\int_0^1|g|^2+\xi\int_0^1(g''\overline{g}'+\overline{g}''g')
+\xi\eta\int_0^1 (g'\overline{g}+\overline{g}'g)+\eta\int_0^1(g''\overline{g}+\overline{g}''g)\\
&=\int_0^1|g''|^2+\xi^2\int_0^1|g'|^2+\eta^2\int_0^1|g|^2+\xi\int_0^1(g'\overline{g}')'+\xi\eta\int_0^1(g\overline{g})'+\eta\Bigl[g'\overline{g}+\overline{g}'g\Bigr]_0^1-2\eta\int_0^1|g'|^2\\
&=\int_0^1|g''|^2+(\xi^2-2\eta)\int_0^1|g'|^2+\eta^2\int_0^1|g|^2+\xi|g'(1)|^2+\xi\eta|g(1)|^2
+\eta\Bigl(g'(1)\overline{g}(1)+\overline{g}'(1)g(1)\Bigr).
\end{align*}
Thus condition \eqref{E:gineq} will be satisfied if and only if
\[
(\xi^2-2\eta)\int_0^1|g'|^2+\eta^2\int_0^1 |g|^2+\xi|g'(1)|^2+\xi\eta|g(1)|^2
+\eta\Bigl(g'(1)\overline{g}(1)+\overline{g}'(1)g(1)\Bigr)\le0.
\]
Now $g=V(g')$, and from \eqref{E:Halmos} we have $\|V\|=2/\pi$, so $\|g\|_2\le (2/\pi)\|g'\|_2$, i.e.,
\[
\int_0^1 |g|^2\le \frac{4}{\pi^2}\int_0^1|g'|^2.
\] 
Thus \eqref{E:gineq} will hold provided that
\[
\Bigl(2\eta-\xi^2-4\eta^2/\pi^2\Bigr)\int_0^1|g'|^2-\xi|g'(1)|^2-\xi\eta|g(1)|^2
-\eta\Bigl(g'(1)\overline{g}(1)+\overline{g}'(1)g(1)\Bigr)\ge0.
\]
Also, by the fundamental theorem of calculus and the Cauchy--Schwarz inequality, we have
\[
|g(1)|^2=|g(1)-g(0)|^2=\Bigl|\int_0^1 g'\Bigr|^2\le \int_0^1|g'|^2.
\]
Hence, provided that
\begin{equation}\label{E:cond1}
2\eta-\xi^2-4\eta^2/\pi^2\ge0,
\end{equation}
the inequality \eqref{E:gineq} will hold if
\[
\Bigl(2\eta-\xi^2-4\eta^2/\pi^2-\xi\eta\Bigr)|g(1)|^2-\xi|g'(1)|^2
-\eta\Bigl(g'(1)\overline{g}(1)+\overline{g}'(1)g(1)\Bigr)\ge0.
\]
By Lemma~\ref{L:completesquare},  for this last inequality to hold,
it suffices that 
\begin{equation}\label{E:cond2}
\left\{
\begin{aligned}
2\eta-\xi^2-4\eta^2/\pi^2-\xi\eta&\ge0,\\
-\xi&\ge0,\\
 (-\xi)\Bigl(2\eta-\xi^2-4\eta^2/\pi^2-\xi\eta\Bigr)&\ge \eta^2.
\end{aligned}
\right.
\end{equation}
Condition \eqref{E:cond1} implies that $\eta\ge0$, and  obviously the second condition in \eqref{E:cond2}
implies that $\xi\le0$. Therefore the first condition in \eqref{E:cond2} is an automatic consequence of 
\eqref{E:cond1}.
Thus \eqref{E:cond1} and \eqref{E:cond2} together are equivalent to the conditions 
\eqref{E:flatnorm}, and we have shown that if they hold, then so does
\eqref{E:gineq}
and consequently $\|I+\xi V+\eta V^2\|\le1$. 
\end{proof}


\section{Numerical range and Crouzeix's conjecture}\label{S:nr}

Let $T$ be a bounded linear operator on a complex Hilbert space $H$.
The \emph{numerical range} of $T$ is defined by
\[
W(T):=\Bigl\{\langle Tx,x\rangle : x\in H,\, \|x\|=1\Bigr\}.
\]
It is a bounded convex set whose closure  contains the spectrum of $T$. 
If further $T$ is a compact operator, then $W(T)$ is compact.

According to a celebrated conjecture of Crouzeix \cite{Cr04}, 
for every operator $T$ and every polynomial $p(z)$ we have
\begin{equation}\label{E:Crouzeix}
\|p(T)\|\le 2\sup_{z\in W(T)}|p(z)|.
\end{equation}
The issue here is the constant~$2$. Crouzeix and Palencia \cite{CP17} showed that
inequality \eqref{E:Crouzeix} holds if one replaces $2$ by the slightly larger constant $(1+\sqrt{2})$.

As mentioned in the introduction, this paper was prompted in part by a question
posed to the first author by Felix Schwenninger as to whether the Crouzeix conjecture
holds for the Volterra operator~$V$. 
To try to answer this, the first step is to identify
the numerical range of $V$. 

\begin{theorem}\label{T:nr}
$W(V)$ is the convex compact set bounded by the vertical segment $[-i/2\pi,\,i/2\pi]$ and the curves
\[
t\mapsto \Bigl(\frac{1-\cos t}{t^2}\Bigr) \pm i\Bigl(\frac{t-\sin t}{t^2}\Bigr) \quad(t\in[0,2\pi]).
\]
\end{theorem}

\begin{figure}[ht]
\begin{center}
\includegraphics[scale=0.4, trim=0 340 100 50, clip]{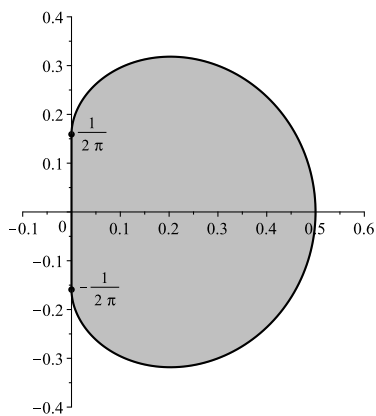}
\caption{Numerical range of $V$} 
\label{F:nr}
\end{center}
\end{figure}

\begin{remark}
This result is folklore. It appears in the middle of a discussion in \cite[p.113--114]{Ha82},
where it is attributed to A.~Brown. We sketch briefly how it may be obtained. 
\end{remark}

\begin{proof}
According to a well-known formula of Lumer \cite[Lemma~12]{Lu61},
if $T$ is a bounded operator on a Hilbert space, then
\[
\sup\Bigl\{\Re z: z\in W(T)\Bigr\}=\lim_{r\to0^+}\frac{\|I+rT\|-1}{r}.
\]
Applying this with $T=e^{i\theta}V$, and using Theorem~\ref{T:nearI}, we deduce that
\[
\sup\Bigr\{\Re z: z\in e^{i\theta}W(V)\Bigr\}=\frac{\sin\theta}{2\theta} \quad(\theta\in[-\pi,\pi]).
\]
This identifies the support function of $W(V)$. It turns out to be the same as the support function
of the set described in the statement of Theorem~\ref{T:nr}.
The details of this calculation can be found for example in \cite[p.105]{KT18}.
As both sets are convex and compact, they must be equal.
\end{proof}

To test whether the Crouzeix conjecture holds for $V$, 
we compute the ratio $\|p(V)\|/\max_{W(V)}|p|$ for various polynomials~$p$,
and check whether this is always bounded above by $2$. The denominator is relatively easy to compute. Indeed, by the maximum modulus principle, $\max_{W(V)}|p|$ is attained on the boundary of $W(V)$, for which we have an explicit parametrization, so this reduces to a one-dimensional maximization problem. The main challenge is to compute the numerator, $\|p(V)\|$. This leads us directly to the problem addressed in the preceding sections of this paper.

We have developed methods to compute $\|p(V)\|$ when $p$ is a linear or real-quadratic polynomial.
It is easy to see that $\|p(V)\|/\max_{W(V)}|p|\le 2$ when $p$ is a linear polynomial, so we shall concentrate
on the case when $p$ is real-quadratic. Though this is a very special case, it is not altogether unreasonable to hope that,
if there is a polynomial $p$ such that $\|p(V)\|/\max_{W(V)}|p|>2$, then there is real-quadratic one with this property.
Indeed, since $\|V^n\|\to0$ very rapidly with $n$, one might expect lower powers of $V$ to dominate in $p(V)$, and since $W(V)$
is symmetric about the $x$-axis, one might expect the same to be true of the roots of $p$. Of course, this is purely heuristic. 

\begin{figure}[ht]
\begin{center}
\includegraphics[scale=0.08, trim= 0 110 0 60]{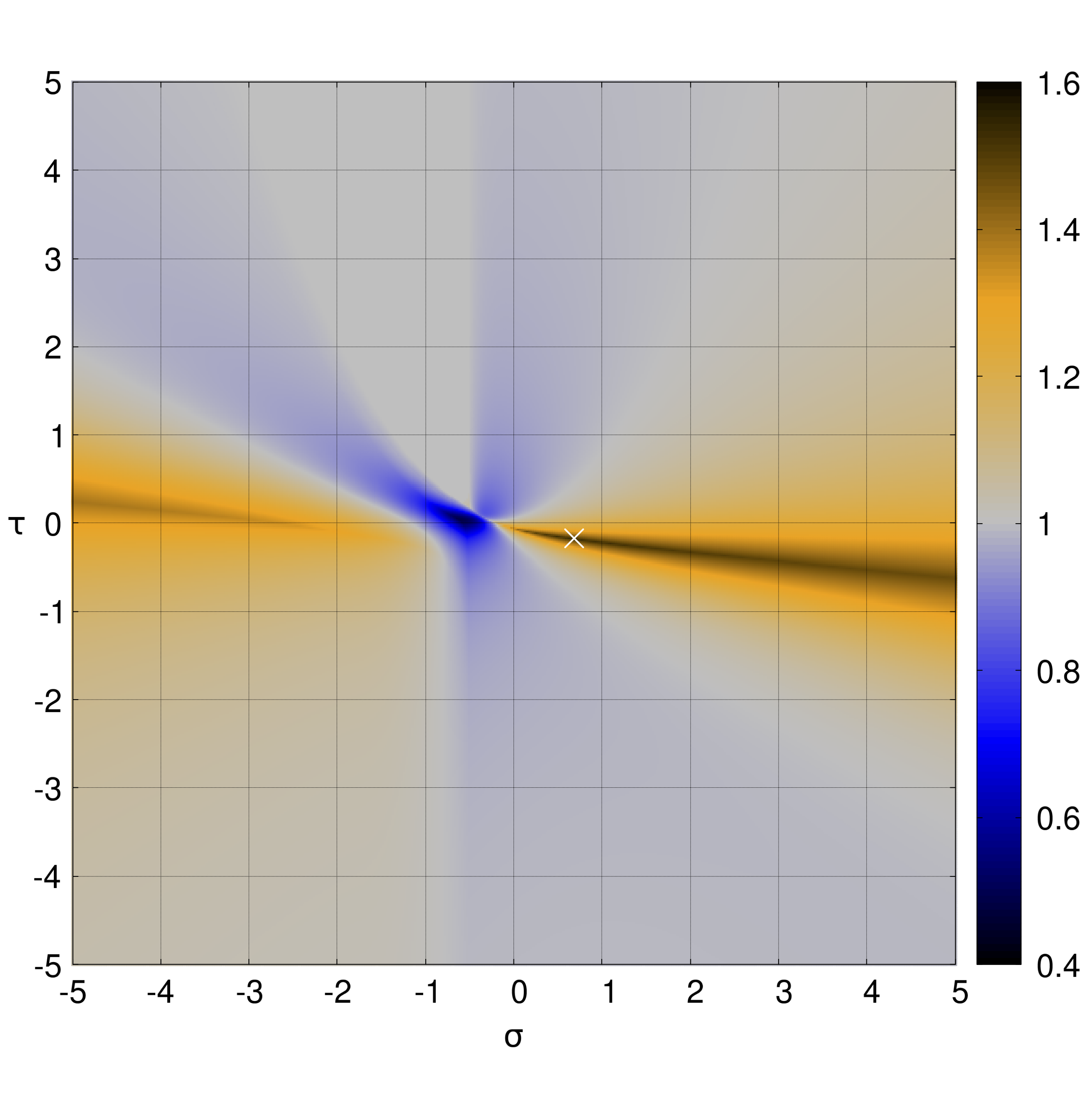}
\caption{Plot of $(\sigma,\tau)\mapsto\|V^2+\sigma V+\tau I\|/\max_{W(V)}|z^2+\sigma z+\tau|$} 
\label{F:Crouzeix}
\end{center}
\end{figure}

Figure~\ref{F:Crouzeix} illustrates the results of our computations.
The largest value that we have found for $\|V^2+\sigma V+\tau I\|/\max_{W(V)}|z^2+\sigma z+\tau|$
is $1.5278$, attained when $(\sigma,\tau)=(0.685,-0.167)$ (marked with a white cross on the figure). In this case,
$\|V^2+\sigma V+\tau I\|=0.6501$ and $\max_{W(V)}|z^2+\sigma z+\tau|=0.4255$.

If we factorize $p(z)=z^2+\sigma z+\tau$ as $p(z)=(z-\alpha)(z-\beta)$, 
then, since the coefficients $\sigma,\tau$ are both real,
either $\alpha,\beta$ are both real, say $(\alpha,\beta)=(x_1,x_2)$,
or they are complex conjugates of one another, say $(\alpha,\beta)=(a+ib,a-ib)$.
Figure~\ref{F:Crouzeix2} shows plots of the values of $\|p(V)\|/\max_{W(V)}|p(z)|$ where $p(z)$ is parametrized  either by $(x_1,x_2)$ (in the case of real roots) or by $(a,b)$ (in the case of complex conjugate roots).
In the second case, the numerical range of $V$ has been superimposed on the plot.
The maximum computed value of $\|p(V)\|/\max_{W(V)}|p(z)|$ (namely $1.5258$) is attained in the case of real roots, 
with $x_1=0.191$ and $x_2=-0.876$ (marked with a white cross on the figure).

\begin{figure}[ht]
\begin{minipage}{0.49\linewidth}
\begin{center}
\includegraphics[scale=0.07, trim=0 100 0 50, clip]{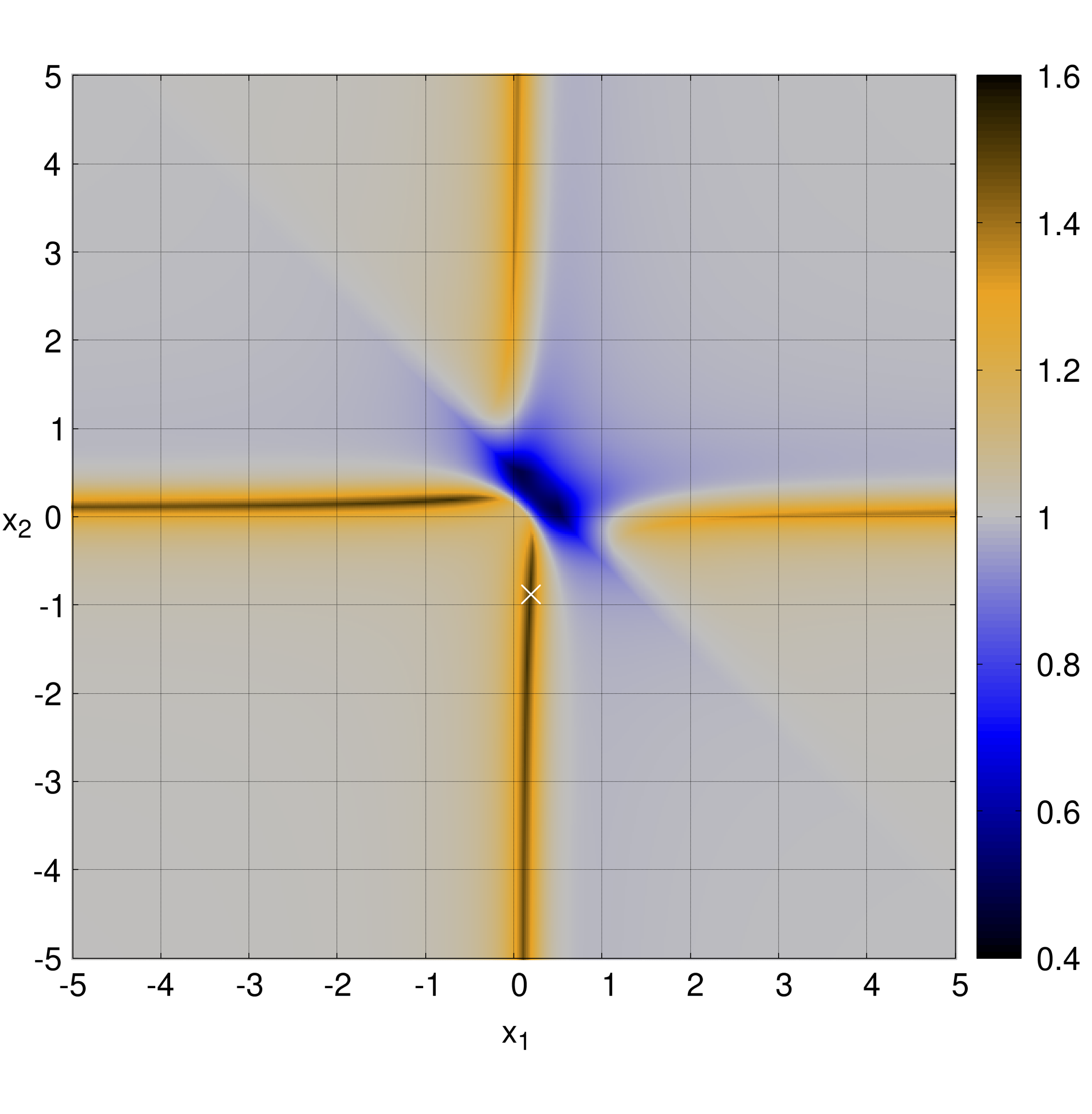}
\end{center}
\end{minipage}
\begin{minipage}{0.49\linewidth}
\begin{center}
\includegraphics[scale=0.075, trim= 0 100 0 50]{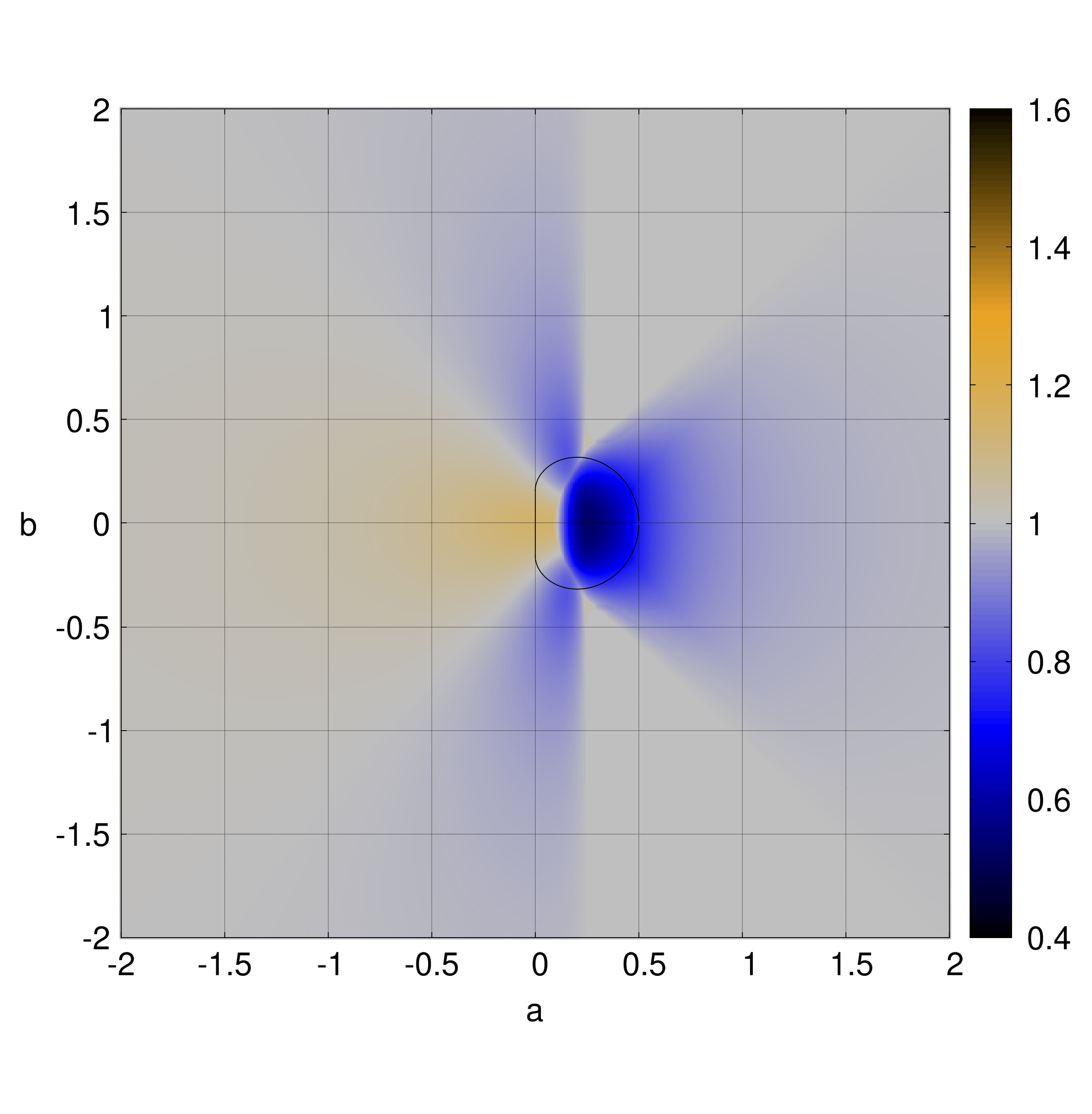}
\end{center}
\end{minipage}
\caption{Plot of $\|p(V)\|/\max_{W(V)}|p(z)|$, where $p$ is parametrized by real roots (left) and complex conjugate roots (right)} 
\label{F:Crouzeix2}
\end{figure}


\section{Conclusion}

Motivated in part by the problem of testing whether the Crouzeix conjecture holds for the Volterra operator $V$,
we have developed methods for computing the operator norm $\|p(V)\|$ when $p(z)$ is a  real-quadratic polynomial.
We obtain a result expressing $\|p(V)\|$ in terms of the largest root of a certain function. In particular, this allows us to recover the exact value of $\|V^2\|$, which was previously known. We also show that $\|I+\xi V+\eta V^2\|=1$ for $(\xi,\eta)$ in a certain neighborhood of $(-1,1)$, a fact that we believe to be new.
Finally, we have performed numerical tests which lend support to the belief that Crouzeix conjecture holds for $V$.
Our computations show that $\|p(V)\|\le 1.53\max_{W(V)}|p(z)|$ whenever $p(z)$ is a real-quadratic polynomial.

\section*{Acknowledgement}
The authors are grateful to the referee for drawing their attention to references \cite{KT18} and \cite{LT10}.

\bibliographystyle{amsplain}
\bibliography{biblist.bib}

\end{document}